\documentclass[11pt]{amsart}

\usepackage{fullpage}

\usepackage[T1]{fontenc}
\usepackage{lmodern}
\usepackage{microtype}
\usepackage{amsmath,amssymb,amsthm,mathtools}
\usepackage{enumitem}
\usepackage[colorlinks=true,linkcolor=blue,citecolor=blue,urlcolor=blue]{hyperref}

\setlist[enumerate]{label=(\roman*),leftmargin=2.2em}
\newtheorem{theorem}{Theorem}[section]
\newtheorem{proposition}[theorem]{Proposition}
\newtheorem{lemma}[theorem]{Lemma}
\newtheorem{corollary}[theorem]{Corollary}
\theoremstyle{definition}

\newtheorem{example}[theorem]{Example}
\newtheorem{question}[theorem]{Question}
\theoremstyle{remark}
\newtheorem{remark}[theorem]{Remark}

\newcommand{\Z}{\mathbb Z}
\newcommand{\F}{\mathbb F}
\newcommand{\R}{\mathbb R}
\newcommand{\T}{\mathbb T}
\newcommand{\vol}{\operatorname{vol}}
\newcommand{\htop}{\operatorname{ht}}
\newcommand{\supp}{\operatorname{supp}}

\title[Scott-type inequalities for trinomial $h^*$-polynomials]
{Scott-type inequalities for trinomial $h^*$-polynomials}
\author{Akiyoshi Tsuchiya}
\address{Department of Information Science, Faculty of Science, Toho University, 2-2-1 Miyama, Funabashi, Chiba 274-8510, Japan}
\email{akiyoshi@is.sci.toho-u.ac.jp}
\date{October 23, 2026}

\hypersetup{
 pdftitle={Scott-type inequalities for trinomial h*-polynomials},
 pdfauthor={Akiyoshi Tsuchiya},
 pdfsubject={Ehrhart theory and volume inequalities for lattice polytopes}
}

\begin{document}

\begin{abstract}
Motivated by Scott's inequality for lattice polygons, we study normalized volume bounds for lattice polytopes $\mathcal{P}$ with $h^*$-polynomials $1+at^k+bt^\ell$, where $a,b\ge1$ and $1\le k<\ell$.
For $\ell>2k$, the coefficients satisfy $a+1\le b$, and consequently $\vol(\mathcal{P})\le2b$.
In the complementary range $k<\ell\le2k$, we establish the dimension-independent bound
\[
  \vol(\mathcal{P})\le(2b+1)(k+\ell).
\]
On the boundary $\ell=2k$ with $b\ge2$,
the sharper Scott-type inequality
$\vol(\mathcal{P})\le4k(b+1)$ holds,
confirming a conjecture of Higashitani, Nill, and the author.
Equality holds if and only if $k=1$ and $a=3b+3$.
Moreover, when $\ell=2k$ and $k,b\ge2$,
we obtain an upper bound depending only on $b$.
\end{abstract}

\maketitle

\section{Introduction}

Let $\mathcal{P}\subset\R^d$ be a $d$-dimensional lattice polytope. Its Ehrhart series has the form
\[
 \sum_{j\ge0}|j\mathcal{P}\cap\Z^d|t^j
 =\frac{h^*_{\mathcal{P}}(t)}{(1-t)^{d+1}},
\]
where $h^*_{\mathcal{P}}(t)=h^*_0+h^*_1t+\cdots+h^*_s t^s$ has nonnegative integer coefficients, $h^*_0=1$, $h^*_s\ne0$, and $s\le d$.
This numerator is called the \emph{$h^*$-polynomial} of $\mathcal{P}$.
We write $\vol(\mathcal{P})=h^*_{\mathcal{P}}(1)$ for its normalized volume, namely $d!$ times its Euclidean volume.
For these standard facts, see \cite[Chapter~3]{BeckRobins}.

A basic problem in Ehrhart theory is to bound normalized volume using limited information about the $h^*$-polynomial.
Haase, Nill, and Payne \cite{HNP} proved a dimension-independent upper bound depending only on the degree and the leading coefficient.
Here we consider trinomials $h^*_{\mathcal{P}}(t)=1+at^k+bt^\ell$, where $a,b,k,\ell$ are positive integers with $k<\ell$.

\begin{theorem}\label{thm:general}
Let $\mathcal{P}$ be a lattice polytope with
\[
 h^*_{\mathcal{P}}(t)=1+at^k+bt^\ell,
 \qquad a,b\ge1,\qquad 1\le k<\ell.
\]
Then the following hold.
\begin{enumerate}
\item If $\ell>2k$, then $a+1\le b$ and
\[
 \vol(\mathcal{P})=a+b+1\le2b.
\]
\item If $k<\ell\le2k$, then
\[
 \vol(\mathcal{P})=a+b+1\le(2b+1)(k+\ell).
\]
\end{enumerate}
\end{theorem}

Part (i) follows from Stanley's inequality \cite{StanleyHilbert}.
Part (ii) treats the complementary range $k<\ell\le2k$,
where that inequality does not bound $a$ in terms of $b$.
The volume bounds also hold when $a=0$, since then
$\vol(\mathcal{P})=b+1$.

For comparison, if $s\ge1$ is the degree and $b$ is the leading coefficient, \cite[Theorem~4.1]{HNP} yields
\[
 \vol(\mathcal{P})\le C_s b^{N_s},
 \qquad N_s=\frac{s^2+19s-4}{2},
\]
where $C_s$ depends only on $s$.
One may take
$C_s=N_s^{N_s}(N_s!)^{N_s+1}(8N_s)^{N_s^2}15^{\,N_s^2 2^{2N_s+1}}$,
by using the second term in the minimum in \cite[Theorem~4.1]{HNP}.
Thus this general estimate has degree $N_s$ in $b$, whereas Theorem~\ref{thm:general}(ii) is linear in $b$ and in $k+\ell$ separately.

The boundary case $\ell=2k$ leads to a Scott-type inequality.
For a lattice polygon with quadratic $h^*$-polynomial $1+at+bt^2$, Scott's theorem gives $a\le3b+3$, except for $(a,b)=(7,1)$ (see \cite{HaaseSchicho,Scott}).
Treutlein extended this bound to lattice polytopes of degree two in arbitrary dimension \cite{Treutlein}.
Motivated by these results, Higashitani, Nill, and the author \cite[Conjecture~1.3]{HNT} conjectured that
\[
 h^*_{\mathcal{P}}(t)=1+at^k+bt^{2k},\quad b\ge2
 \quad\Longrightarrow\quad
 \vol(\mathcal{P})\le4k(b+1).
\]
Our second theorem proves this Scott-type conjecture and determines when equality can occur.

\begin{theorem}\label{thm:critical}
Let $\mathcal{P}$ be a lattice polytope with
\[
 h^*_{\mathcal{P}}(t)=1+at^k+bt^{2k},
 \qquad a\ge1,\quad b\ge2,\quad k\ge1.
\]
Then
\[
 \vol(\mathcal{P})=a+b+1\le4k(b+1).
\]
If $k\ge2$, the inequality is strict, so
\[
 \vol(\mathcal{P})\le4k(b+1)-1.
\]
In particular, equality in $\vol(\mathcal{P})\le4k(b+1)$ holds if and only if $k=1$ and $a=3b+3$.
\end{theorem}

For every $b\ge2$, the rectangle $[0,2]\times[0,b+1]$ attains equality when $k=1$.
Indeed, its Ehrhart polynomial is $(2j+1)((b+1)j+1)$, so its $h^*$-polynomial is $1+(3b+3)t+bt^2$.
On the boundary $\ell=2k$, Theorem~\ref{thm:general}(ii) gives $3k(2b+1)$. Theorem~\ref{thm:critical} improves this bound by $k(2b-1)$.

The assumption $b\ge2$ in Theorem~\ref{thm:critical} is essential.
The case $b=1$ with $k\ge2$ is covered by the classification in
\cite[Theorem~3.1]{HNT}. See also \cite[Corollary~1.1]{HNT}.
We record the following consequence for comparison.

\begin{proposition}[{\cite[Theorem~3.1]{HNT}}]\label{prop:b-one}
Let $k\ge2$ and let $\mathcal{P}$ be a lattice polytope with
\[
 h^*_{\mathcal{P}}(t)=1+at^k+t^{2k},\qquad a\ge1.
\]
Then either $\vol(\mathcal{P})\le8k$, or
\[
 \vol(\mathcal{P})=9k
 \qquad\text{and}\qquad
 k=3^s\text{ for some integer }s\ge1.
\]
The second alternative is attained for every $s\ge1$
by the ternary family in \cite[Theorem~3.1(c)]{HNT}.\footnote{The assertion in \cite[Section~1.4]{HNT} that equality in $\vol(\mathcal{P})\le9k$ implies $k=1$ should also include the ternary family in \cite[Theorem~3.1(c)]{HNT}, which attains equality for $k=3^s$, $s\ge1$.}
\end{proposition}

\begin{remark}\label{rem:general-sharp-b-one}
Theorem~\ref{thm:general}(ii) is sharp when $b=1$.
The triangle with vertices $(0,0)$, $(3,0)$, and $(0,3)$ has $h^*$-polynomial $1+7t+t^2$ and normalized volume $9=3(1+2)$.
For $k=3^s$ with $s\ge1$, the ternary family in Proposition~\ref{prop:b-one} has $h^*$-polynomial $1+(9k-2)t^k+t^{2k}$ and normalized volume $9k=3(k+2k)$.
Thus the factor $2b+1$ cannot be uniformly reduced at $b=1$, whereas Theorem~\ref{thm:critical} improves the boundary estimate once $b\ge2$.
The optimal bounds for $k<\ell<2k$, and for $\ell=2k$ with $k\ge2$ and $b\ge2$, are not determined here.
\end{remark}

For $b\ge2$, the dependence on $k$ can be removed at the cost of a much larger dependence on $b$.
Proposition~\ref{prop:b-only} gives $\vol(\mathcal{P})\le(8b^2)^b$ when $\ell=2k$ and $k\ge2$.
In particular, the normalized volume is bounded as $k$ varies with $b\ge2$ fixed, unlike the case $b=1$.
Example~\ref{ex:replicated-triangles} realizes normalized volume $4(b+1)$ for every $k,b\ge1$.
In two of the structural cases, Proposition~\ref{prop:structural-linear} gives bounds independent of $k$ and linear in $b$. These are $4(b+1)$ in the cyclic case and $8b-2$ in the simplex-code case with $p\ge5$.
These estimates lead to Question~\ref{que:uniform-linear}, which asks whether a bound linear in $b$ holds uniformly in $k$.

The paper is organized as follows.
Section~\ref{sec:group} recalls the finite-group description of lattice simplices and establishes the results on constant-height subgroups.
In Section~\ref{sec:general-bound}, we derive the coset-counting formulas and prove Theorem~\ref{thm:general}.
Section~\ref{sec:critical} treats the case $\ell=2k$, proves the structural linear bounds and Theorem~\ref{thm:critical}, and concludes with the bound depending only on $b$, the family of examples, and a question about a uniform linear bound.

\subsection*{Acknowledgment}
This work was supported by JSPS KAKENHI Grant Numbers
JP22K13890 and JP26K00618.
\section{Finite groups and constant-height subgroups}\label{sec:group}

Throughout this section, let $\Delta\subset\R^d$ be a $d$-dimensional lattice simplex with
\[
 h^*_{\Delta}(t)=1+at^k+bt^\ell,
 \qquad a,b\ge1,\quad 2\le k<\ell\le2k.
\]
The restriction to simplices is automatic for these
$h^*$-polynomials.
Indeed, any $d$-dimensional lattice polytope with $h^*_1=0$
has exactly $d+1$ lattice points and is therefore an empty simplex.
Thus the only lattice points of $\Delta$ are its vertices.

Set $\T=\R/\Z$, written additively, and represent its elements by numbers in $[0,1)$.
For every positive integer $N$, write
\[
 \T[N]:=\{z\in\T:Nz=0\}
 =\left\{\frac{j}{N}+\Z:j=0,\ldots,N-1\right\},
\]
the unique subgroup of $\T$ of order $N$.
For a prime $p$, we identify $\T[p]$ with the additive group of $\F_p$ via $j/p+\Z\mapsto j\bmod p$.

For ordered vertices $v_0,\ldots,v_d$ of $\Delta$, define
\[
 G=\Lambda_\Delta
 =\left\{x\in[0,1)^{d+1}:
 \sum_{i=0}^d x_i(v_i,1)\in\Z^{d+1}\right\}.
\]
This is a finite subgroup of $\T^{d+1}$, with addition taken coordinatewise modulo $\Z$.
For $x\in G$, put
\[
 \htop(x)=\sum_{i=0}^d x_i,
 \qquad \supp(x)=\{i\in\{0,\ldots,d\}:x_i\ne0\}.
\]
The finite-group description of lattice simplices gives
\[
 h^*_j=|\{x\in G:\htop(x)=j\}|,
 \qquad |G|=\vol(\Delta)=a+b+1;
\]
see \cite[Theorem~2.3 and Proposition~2.6]{BHCode}.
Thus every nonzero element of $G$ has height $k$ or $\ell$.

For $x\in G$ and $i\in\{0,\ldots,d\}$, the coordinate $(-x)_i$ equals $0$ when $x_i=0$ and $1-x_i$ otherwise.
Consequently,
\begin{equation}\label{eq:support-height}
 |\supp(x)|=\htop(x)+\htop(-x).
\end{equation}
In particular,
\begin{equation}\label{eq:min-support}
 |\supp(x)|\ge2k\qquad(x\in G\setminus\{0\}).
\end{equation}
Define
\[
 B=\{x\in G:\htop(x)=\ell\},\qquad |B|=b,
\]
and
\[
 X=\{0\}\cup\{x\in G:\htop(x)=\htop(-x)=k\}.
\]

\begin{lemma}\label{lem:cyclic-closed}
If $x\in X$, then $\langle x\rangle\subseteq X$.
Moreover, the complement of $X$ satisfies
\[
 G\setminus X=B\cup\{-x:x\in B\},
 \qquad |G\setminus X|\le2b.
\]
\end{lemma}

\begin{proof}
Let $x\in X\setminus\{0\}$.
Then $|\supp(x)|=2k$ by \eqref{eq:support-height}, and $\supp(jx)\subseteq\supp(x)$ for every integer $j$.
If a nonzero multiple $jx$ had height $\ell$, then
\[
 |\supp(jx)|=\htop(jx)+\htop(-jx)\ge k+\ell>2k,
\]
a contradiction.
Thus all nonzero multiples of $x$, including their inverses, have height $k$.
The remaining assertions follow because an element lies outside $X$ exactly when it or its inverse belongs to $B$.
\end{proof}

The following lemma supplies a subgroup of controlled index.
Groups are written additively, with identity $0$.
We write $H\le A$ to mean that $H$ is a subgroup of $A$, and $[A:H]$ for its index.

\begin{lemma}\label{lem:avoidance}
Let $A$ be a finite abelian group, and let $X\subseteq A$ satisfy
\[
 0\in X,\qquad x\in X\Longrightarrow\langle x\rangle\subseteq X.
\]
Then there is a subgroup $H\le A$ such that
\[
 H\subseteq X,\qquad [A:H]\le|A\setminus X|+1.
\]
\end{lemma}

\begin{proof}
We argue by induction on $|A|$.
If $|A|=1$, then $X=A=\{0\}$, and we take $H=A$.
Assume $|A|>1$ and that the assertion holds for all smaller finite abelian groups and all subsets satisfying the stated hypotheses.
Put $Y=A\setminus X$.
We may assume $Y\ne\emptyset$, since otherwise $H=A$ suffices.

Choose a prime $p$ dividing $|A|$, and consider the homomorphism
\[
 \mu_p:A\longrightarrow A,\qquad a\longmapsto pa.
\]
Set
\[
 A_1=\operatorname{im}\mu_p,\quad K_p=\ker\mu_p,
 \quad P=|K_p|,\quad Y_1=Y\cap A_1.
\]
Cauchy's theorem gives $P>1$.
The first isomorphism theorem therefore yields
\[
 |A_1|=|A|/P<|A|,\qquad [A:A_1]=P.
\]
For $y\in A_1$, choose $a_0\in A$ with $\mu_p(a_0)=y$.
Then
\[
 \{a\in A:\mu_p(a)=y\}=a_0+K_p,
\]
so the equation $\mu_p(a)=y$ has exactly $P$ solutions.
It follows that $|\mu_p^{-1}(Y_1)|=P|Y_1|$.
On the other hand, for every $a\in A$, the hypothesis on $X$ implies
\[
 pa\in Y\Longrightarrow a\in Y.
\]
Indeed, $a\in X$ would imply $pa\in\langle a\rangle\subseteq X$.
Thus $\mu_p^{-1}(Y_1)\subseteq Y$, and the disjoint decomposition
\[
 Y=\mu_p^{-1}(Y_1)\sqcup\bigl(Y\setminus\mu_p^{-1}(Y_1)\bigr)
\]
gives
\[
 |Y|=P|Y_1|+M,\qquad M=|Y\setminus\mu_p^{-1}(Y_1)|.
\]

Suppose first that $Y_1\ne\emptyset$.
The subset $X_1=X\cap A_1$ contains $0$.
For $x\in X_1$, the subgroup $\langle x\rangle$ is contained in both $X$ and $A_1$, hence in $X_1$.
Moreover, $A_1\setminus X_1=Y_1$.
Induction applied to $(A_1,X_1)$ gives $H_1\le A_1$ with
\[
 H_1\subseteq X_1,\qquad [A_1:H_1]\le|Y_1|+1.
\]
If $M+1\ge P$, take $H=H_1$.
Its index then satisfies
\[
 [A:H_1]=P[A_1:H_1]
 \le P(|Y_1|+1)\le P|Y_1|+M+1=|Y|+1.
\]

Otherwise, $M+1<P$.
Let $A_2=\mu_p^{-1}(H_1)$.
Since $Y_1\ne\emptyset$ and $H_1\cap Y_1=\emptyset$, the subgroup $H_1$ is proper in $A_1$.
Surjectivity of $\mu_p:A\to A_1$ implies that $A_2$ is proper in $A$.
If $a\in Y\cap A_2$, then $\mu_p(a)\in H_1\subseteq X_1$.
Since $X_1\cap Y_1=\emptyset$, we have $a\notin\mu_p^{-1}(Y_1)$.
Therefore one has
\[
 Y\cap A_2\subseteq Y\setminus\mu_p^{-1}(Y_1),
 \qquad |Y\cap A_2|\le M.
\]
As above, $X_2=X\cap A_2$ contains $0$ and contains $\langle x\rangle$ for every $x\in X_2$. Moreover, $A_2\setminus X_2=A_2\cap Y$.
Applying induction to $(A_2,X_2)$ gives $H\le A_2$ with $H\subseteq X_2$ and $[A_2:H]\le M+1$.
The surjective homomorphism
\[
 A\longrightarrow A_1/H_1,\qquad a\longmapsto\mu_p(a)+H_1
\]
has kernel $A_2$, so $A/A_2\cong A_1/H_1$.
Consequently,
\[
 \begin{aligned}
 [A:H]&=[A:A_2][A_2:H]\\
 &\le(|Y_1|+1)(M+1)\\
 &\le P|Y_1|+M+1=|Y|+1,
 \end{aligned}
\]
where the last inequality uses $M+1\le P$.

It remains to treat $Y_1=\emptyset$.
Then $A_1\subseteq X$.
Let $\pi:A\to V=A/A_1$ be the quotient map.
For every $a\in A$, one has $p(a+A_1)=0+A_1$.
Since $V$ is finite, it is a finite-dimensional $\F_p$-vector space.
Put $\overline Y=\pi(Y)$.
The zero element $0+A_1$ of $V$ does not belong to $\overline Y$. Indeed, if $\pi(y)=0+A_1$ for some $y\in Y$, then $y\in Y\cap A_1=Y_1$, a contradiction.

We claim that $\overline Y$ is stable under multiplication by $\F_p^\times$.
Fix $y\in Y$ and $c\in\F_p^\times$, and put $N=|\langle y\rangle|$.
Since $0\in X$ and $y\in Y=A\setminus X$, we have $y\ne0$ and $N>1$.
Let $c_0\in\{1,\ldots,p-1\}$ be the integer representing $c$, and let $q_1,\ldots,q_s$ be the distinct prime divisors of $N$ other than $p$.
If $s=0$, take $n=c_0$.
If $s\ge1$, the Chinese remainder theorem gives an integer $n$ such that
\[
 n\equiv c_0\pmod p,\qquad n\equiv1\pmod{q_i}\quad(1\le i\le s).
\]
In either case, $n\equiv c_0\pmod p$ and no prime divisor of $N$ divides $n$.
Thus $\gcd(n,N)=1$, and $\langle ny\rangle=\langle y\rangle$.
If $ny\in X$, the hypothesis on $X$ would give $y\in X$.
Hence $ny\in Y$, and $c\pi(y)=\pi(ny)\in\overline Y$.

The zero subspace is disjoint from $\overline Y$.
Choose a subspace $U\le V$ of maximum dimension among those disjoint from $\overline Y$.
Let $v\in V\setminus U$.
By maximality, $(U+\F_pv)\cap\overline Y\ne\emptyset$.
Choose $u\in U$ and $\alpha\in\F_p$ such that $u+\alpha v\in\overline Y$.
Since $U\cap\overline Y=\emptyset$, we have $\alpha\ne0$.
The stability of $\overline Y$ under multiplication by $\F_p^\times$ gives
\[
 v+\alpha^{-1}u\in\overline Y\cap(v+U).
\]
Thus every nonzero coset $v+U$ contains an element of $\overline Y$.
There are $|V/U|-1$ such cosets, and distinct cosets are disjoint.
Consequently,
\[
 |V/U|-1\le|\overline Y|\le|Y|.
\]
Finally, set $H=\pi^{-1}(U)$.
Then $H\cap Y=\emptyset$, since otherwise $U\cap\overline Y\ne\emptyset$.
The composite map
\[
  \varphi:A\longrightarrow V/U,\qquad
  a\longmapsto\pi(a)+U
\]
is surjective, and
\[
  \ker\varphi
  =\{a\in A:\pi(a)\in U\}
  =\pi^{-1}(U)=H.
\]
Hence the first isomorphism theorem gives $A/H\cong V/U$.
Therefore $H\subseteq X$ and $[A:H]=|V/U|\le|Y|+1$.
\end{proof}

\begin{remark}\label{rem:avoidance-sharp}
The bound is sharp. Indeed, if $X=\{0\}$ in a finite abelian group $A$, the only possible subgroup is $\{0\}$, and its index is $|A|=|A\setminus X|+1$.
\end{remark}

Apply Lemma~\ref{lem:avoidance} to $A=G$ and the set $X$ above, and denote the resulting subgroup by $H_0$.
Fix a subgroup $H\le G$ of maximum order among those satisfying $H_0\subseteq H\subseteq X$.
Such a choice exists because $G$ is finite.
Then $H$ is inclusion-maximal among all subgroups contained in $X$, and enlarging $H_0$ does not increase its index.
Write
\[
 m=|H|,\qquad q=[G:H].
\]
Since $|B|=b\ge1$ and $B\cap X=\emptyset$, the subgroup $H\subseteq X$ is proper in $G$.
Together with Lemmas~\ref{lem:avoidance} and~\ref{lem:cyclic-closed}, this gives
\begin{equation}\label{eq:index-range}
 2\le q=[G:H]\le[G:H_0]\le|G\setminus X|+1\le2b+1.
\end{equation}
If $m\le2k$, then \eqref{eq:index-range} yields
\[
 |G|=mq\le2k(2b+1)<(2b+1)(k+\ell).
\]
We now assume $m>2k$ and describe the possible subgroups $H$.

Let
\[
 S=\{i\in\{0,\ldots,d\}:\text{some }h\in H\text{ satisfies }h_i\ne0\},
 \qquad n=|S|.
\]
After deleting the coordinates outside $S$, the group $H$ corresponds to a lattice simplex of dimension $n-1$ with $h^*$-polynomial $1+(m-1)t^k$.
This follows from the correspondence between lattice simplices and finite subgroups with integral height \cite[Theorem~2.3 and Proposition~2.6]{BHCode}.
No coordinate of the resulting group vanishes identically, so the simplex is not a lattice pyramid by \cite[Proposition~2.5]{BHCode}.
A lattice pyramid is a polytope equivalent, under an affine lattice isomorphism, to $\operatorname{conv}(\mathcal{P}'\times\{0\},(0,\ldots,0,1))$ for some lattice polytope $\mathcal{P}'$.
We use the structural part of the Batyrev--Hofscheier classification, namely \cite[Theorem~3.11]{BHCode}.
See also \cite[Theorems~2.3 and~2.5]{HNT} for a summary of the binomial classification.

\begin{proposition}\label{prop:BH}
Assume $m>2k$. Then exactly one of the following holds.
\begin{enumerate}
\item\label{item:cyclic}
One has $n=2k$.
The group $H$ is cyclic of order $m$, and every coordinate projection $H\to\T$, $h\mapsto h_i$ for $i\in S$, is injective.
\item\label{item:code}
One has $n>2k$.
There are a prime $p$ and an integer $r\ge2$ such that
\[
 H\cong(\Z/p\Z)^r,\qquad m=p^r.
\]
Moreover, the number $\lambda=2k/p^{r-1}$ is a positive integer and satisfies
\begin{equation}\label{eq:code-relations}
 n=\lambda\frac{p^r-1}{p-1},
 \qquad \frac{m}{2k}=\frac p\lambda.
\end{equation}
Under the identification $\T[p]\cong\F_p$, every coordinate projection on $S$ is a nonzero $\F_p$-linear functional on $H$.
If $p$ is odd, then $\lambda$ is even.
\end{enumerate}
\end{proposition}

\begin{proof}
Every nonzero element of $H$ has support of size $2k$, so $n\ge2k$.
If $n=2k$, no nonzero element vanishes at any coordinate in $S$.
Thus each coordinate projection is injective.
Since a finite subgroup of $\T$ is cyclic, $H$ is cyclic.

Suppose $n>2k$.
The simplex associated with $H$ has dimension $n-1$, is not a lattice pyramid, and has binomial $h^*$-polynomial of degree $1<k<n/2$.
By \cite[Theorem~3.11]{BHCode}, one has $H\cong(\Z/p\Z)^r$ for some prime $p$ and some $r\ge1$.
Regard $H$ as an $\F_p$-vector space.
For each $i\in S$, let $\pi_i:H\to\T$ be the coordinate projection $h\mapsto h_i$.
Since $ph=0$ for every $h\in H$, we have
\[
 p\pi_i(h)=\pi_i(ph)=0\qquad(h\in H).
\]
Hence its image lies in $\T[p]\cong\F_p$.
Under this identification, $\pi_i$ is linear because it is additive and multiplication by a scalar in $\F_p$ is repeated addition.
By the definition of $S$, the linear map $\pi_i$ is nonzero.
Since its codomain $\F_p$ is one-dimensional, $\pi_i$ is surjective.
Hence the first isomorphism theorem gives
\[
 H/\ker\pi_i\cong\F_p.
\]
Consequently,
\[
 |\ker\pi_i|=\frac{|H|}{p}=\frac mp.
\]
Thus, for each $i\in S$, exactly $m/p$ elements of $H$ have $i$th coordinate zero, and $m-m/p$ have $i$th coordinate nonzero.
Counting the pairs $(h,i)$ with $h\in H$, $i\in S$, and $h_i\ne0$ in two ways, we obtain
\[
 \begin{aligned}
 2k(m-1)
 &=\sum_{h\in H}|\supp(h)|\\
 &=\sum_{i\in S}|\{h\in H:h_i\ne0\}|\\
 &=n\left(m-\frac mp\right).
 \end{aligned}
\]
Here the first equality uses $|\supp(h)|=2k$ for every nonzero $h\in H$.
Since $m=p^r$, this becomes
\[
 p^{r-1}n=2k(1+p+\cdots+p^{r-1}).
\]
The case $r=1$ would give $n=2k$, so $r\ge2$.
The factor $1+p+\cdots+p^{r-1}$ is relatively prime to $p$. Consequently, $p^{r-1}\mid2k$.
Setting $\lambda=2k/p^{r-1}$ gives \eqref{eq:code-relations}.
Finally, when $p$ is odd, $p^{r-1}$ is odd, so $\lambda p^{r-1}=2k$ implies that $\lambda$ is even.
\end{proof}

We call cases (i) and (ii) of Proposition~\ref{prop:BH} the \emph{cyclic case} and the \emph{simplex-code case}, respectively.
The latter terminology comes from the fact that the associated constant-weight linear code is equivalent, up to coordinate permutations and multiplication of coordinates by nonzero scalars, to a repeated simplex code (see \cite[Theorem~4.7]{BHCode}).
In the simplex-code case, $m>2k$ gives
\begin{equation}\label{eq:lambda-bound}
 \lambda<p.
\end{equation}

\section{The general volume bound}\label{sec:general-bound}

Let $\mathcal{P}\subset\R^d$ be a $d$-dimensional lattice polytope with
\[
 h^*_{\mathcal{P}}(t)=1+at^k+bt^\ell,
 \qquad a,b\ge1,\quad 1\le k<\ell.
\]
This section proves Theorem~\ref{thm:general}.
The argument uses the following inequality of Stanley
\cite{StanleyHilbert}.
For a combinatorial proof and related refinements, see \cite[Theorem~2.19 and Remark~1.3]{Stapledon}.

\begin{proposition}[Stanley]\label{prop:stanley}
Let $\mathcal{P}$ be a lattice polytope with $h^*_{\mathcal{P}}(t)=\sum_{j=0}^s h^*_j t^j$ and $h^*_s\ne0$.
Then, for every $0\le i\le s$, one has
\[
 h^*_0+h^*_1+\cdots+h^*_i
 \le h^*_s+h^*_{s-1}+\cdots+h^*_{s-i}.
\]
\end{proposition}

\begin{corollary}\label{cor:large-gap}
If $h^*_{\mathcal{P}}(t)=1+at^k+bt^\ell$ with $b\ge1$ and $\ell>2k$, then $a+1\le b$ and $\vol(\mathcal{P})\le2b$.
\end{corollary}

\begin{proof}
Apply Proposition~\ref{prop:stanley} with $s=\ell$ and $i=k$.
Since $\ell-k>k$, the two sides are $1+a$ and $b$, respectively.
\end{proof}

Suppose $k<\ell\le2k$. If $k=1$, then $\ell=2$.
Treutlein's theorem \cite{Treutlein} gives either $(a,b)=(7,1)$ or $a\le3b+3$.
The exceptional pair gives
\[
 \vol(\mathcal{P})=9=(2b+1)(k+\ell).
\]
In all other cases, we obtain
\[
 \vol(\mathcal{P})\le4b+4\le3(2b+1)=(2b+1)(k+\ell).
\]

For the rest of this section, assume $2\le k<\ell\le2k$.
Since $h^*_1=0$, $\mathcal{P}$ is an empty lattice simplex.
Write $\Delta=\mathcal{P}$ and retain the notation of Section~\ref{sec:group}, including the subgroup $H$ fixed there.
If $m\le2k$, then \eqref{eq:index-range} gives $|G|\le2k(2b+1)<(2b+1)(k+\ell)$, which proves the required bound.
Hence we assume $m>2k$, so the two cases of Proposition~\ref{prop:BH} apply.

Put $Q=G/H$, so $|Q|=q$.
The identity of $Q$ is the coset $H$. We therefore write $C\ne H$ for a nonzero coset.
For a coset $C\in Q$, choose a representative $x\in G$ with $C=x+H$.
Define $-C=(-x)+H$ and
\[
 L(C)=|C\cap B|+|(-C)\cap B|\qquad(C\ne H).
\]
This is a sum of two cardinalities, not the cardinality of a union.
In particular, if $C=-C$, each element of $C\cap B$ is counted twice.
Since $B\cap H=\emptyset$ and the map $C\mapsto -C$ is a bijection on $Q\setminus\{H\}$,
\begin{equation}\label{eq:L-sum}
 \sum_{C\in Q\setminus\{H\}}L(C)=2b.
\end{equation}
For any $C\in Q$ and a representative $x\in G$ with $C=x+H$, put
\[
 \nu(C)=|\{i\notin S:x_i\ne0\}|.
\]
This is independent of the representative, since every element of $H$ vanishes outside $S$.
We refer to coordinates in $S$ and outside $S$ as internal and external coordinates, respectively.
Let
\[
 \rho:G\longrightarrow\T^{\{0,\ldots,d\}\setminus S},
 \qquad x\longmapsto(x_i)_{i\notin S}
\]
be the projection onto the external coordinates.
Since every element of $H$ vanishes outside $S$, this induces a homomorphism
\[
 \overline\rho:Q=G/H\longrightarrow\T^{\{0,\ldots,d\}\setminus S},
 \qquad x+H\longmapsto(x_i)_{i\notin S}.
\]

For every $C\in Q\setminus\{H\}$, we have
\begin{equation}\label{eq:master-support}
 \begin{aligned}
 \sum_{y\in C}|\supp(y)|
 &=\sum_{y\in C}\bigl(\htop(y)+\htop(-y)\bigr)\\
 &=2mk+(\ell-k)L(C).
 \end{aligned}
\end{equation}
Indeed, every element of $C$ has height $k$ or $\ell$, and exactly $|C\cap B|$ elements have height $\ell$.
Hence
\[
 \sum_{y\in C}\htop(y)
 =k\bigl(m-|C\cap B|\bigr)+\ell|C\cap B|
 =mk+(\ell-k)|C\cap B|.
\]
Since the map $y\mapsto -y$ is a bijection from $C$ to $-C$, applying the same formula to $-C$ and adding the two equalities proves \eqref{eq:master-support}, also when $C=-C$.

\begin{lemma}\label{lem:external-injective}
If $|S|<k+\ell$, then $\overline\rho$ is injective.
In particular, $\nu(C)\ge1$ for every nonzero coset $C\in Q$.
\end{lemma}

\begin{proof}
Set
\[
 K=\{x\in G:\supp(x)\subseteq S\}.
\]
This is a subgroup of $G$ containing $H$, since $H$ vanishes outside $S$.
If $x\in K\setminus X$, then either $x$ or $-x$ has height $\ell$, and therefore
\[
 |\supp(x)|\ge k+\ell>|S|,
\]
a contradiction. Thus $K\subseteq X$, and the inclusion-maximality of $H$ among subgroups contained in $X$ gives $K=H$.

By the definition of $\rho$, its kernel is $K$.
Consequently,
\[
 \ker\overline\rho=K/H=\{H\}.
\]
Hence $\overline\rho$ is injective.
\end{proof}

\begin{proposition}\label{prop:cyclic-count}
Assume that $H$ is as in Proposition~\ref{prop:BH}(i). For $C=x+H\in Q$ with $x\in G$, put
\[
 \tau(C)=|\{i\in S:mx_i\ne0\text{ in }\T\}|.
\]
Then, for every $C\ne H$, one has
\begin{equation}\label{eq:cyclic-L-general}
 (\ell-k)L(C)=m\nu(C)+\tau(C)-2k.
\end{equation}
\end{proposition}

\begin{proof}
Since $mh=0$ for every $h\in H$, replacing a representative $x\in G$ by $x+h$ leaves every $mx_i$ unchanged.
Thus $\tau(C)$ is independent of the representative.
Fix $C\in Q\setminus\{H\}$ and choose $x\in G$ with $C=x+H$.
For $i\in S$, the coordinate map $H\to\T$ is injective and therefore has image $\T[m]$.
Thus the equation $x_i+h_i=0$ has exactly one solution $h\in H$ when $mx_i=0$, and no solution otherwise.
The total contribution of the internal coordinates to the support sum is
\[
 (2k-\tau(C))(m-1)+\tau(C)m=2k(m-1)+\tau(C).
\]
The external coordinates contribute $m\nu(C)$.
Comparing with \eqref{eq:master-support} proves \eqref{eq:cyclic-L-general}.

\end{proof}

\begin{proposition}\label{prop:code-count}
Assume that $H$ is as in Proposition~\ref{prop:BH}(ii). For $C=x+H\in Q$ with $x\in G$, put
\[
 \tau(C)=|\{i\in S:px_i\ne0\text{ in }\T\}|.
\]
Then, for every $C\ne H$, one has
\begin{equation}\label{eq:code-L-general}
 (\ell-k)L(C)=m\nu(C)+\frac mp\tau(C)-2k.
\end{equation}
\end{proposition}

\begin{proof}
Since $ph=0$ for every $h\in H$, the definition of $\tau(C)$ is independent of the representative.
Fix $C\in Q\setminus\{H\}$ and choose $x\in G$ with $C=x+H$.
For each $i\in S$, the coordinate projection $\pi_i:H\to\T[p]\cong\F_p$ is a nonzero linear functional.
It is therefore surjective, with kernel of order $m/p$.
If $px_i=0$, the solutions $h\in H$ of $h_i=-x_i$ form one coset of $\ker\pi_i$, so there are exactly $m/p$ solutions.
If $px_i\ne0$, there are no solutions, since $h_i\in\T[p]$ for all $h\in H$.
The internal coordinates therefore contribute
\[
 \begin{aligned}
 &(n-\tau(C))\left(m-\frac mp\right)+\tau(C)m\\
 &\qquad=n\left(m-\frac mp\right)+\frac mp\tau(C)
 =2k(m-1)+\frac mp\tau(C),
 \end{aligned}
\]
where the last equality uses $n(m-m/p)=2k(m-1)$ from the proof of Proposition~\ref{prop:BH}.
Adding the external contribution $m\nu(C)$ and using \eqref{eq:master-support} gives the formula.
\end{proof}

The two coset formulas yield the following uniform estimate.

\begin{lemma}\label{lem:uniform-coset}
For every nonzero coset $C\in Q$, one has
\begin{equation}\label{eq:uniform-coset}
 (\ell-k)L(C)\ge m-2k.
\end{equation}
In particular, $L(C)>0$.
\end{lemma}

\begin{proof}
Fix $C\in Q\setminus\{H\}$.
In the cyclic case, $|S|=2k<k+\ell$, so Lemma~\ref{lem:external-injective} gives $\nu(C)\ge1$. The claimed inequality then follows from \eqref{eq:cyclic-L-general}.
In the simplex-code case, \eqref{eq:code-L-general} gives the same conclusion if $\nu(C)\ge1$.
It remains to consider the simplex-code case with $\nu(C)=0$.
Choose a representative $x\in G$ with $C=x+H$.
The element $px\in G$ is nonzero. Otherwise $\tau(C)=0$, and the right-hand side of \eqref{eq:code-L-general} would equal $-2k$, a contradiction.
Since $x$ vanishes outside $S$, \eqref{eq:min-support} gives
\[
 |\supp(px)|=\tau(C)\ge2k.
\]
By Proposition~\ref{prop:BH}(ii), we have $r\ge2$.
Since $\lambda$ is a positive integer, it follows that
\[
 2k=\lambda p^{r-1}\ge p.
\]
Consequently,
\[
 (\ell-k)L(C)\ge\frac mp\,2k-2k\ge m-2k.
\]
The positivity of $L(C)$ follows from $m>2k$.
\end{proof}

\begin{proof}[Proof of Theorem~\ref{thm:general}]
Part (i), the case $k=1$ of part (ii), and the case $m\le2k$ were settled above.
It remains to prove part (ii) with $2\le k<\ell\le2k$ and $m>2k$.
By \eqref{eq:index-range}, we have $q\ge2$.
Summing \eqref{eq:uniform-coset} and using \eqref{eq:L-sum} gives
\[
 (q-1)(m-2k)\le2b(\ell-k).
\]
Therefore
\[
 m\le2k+\frac{2b(\ell-k)}{q-1}.
\]
Define
\[
 f(u)=2ku+2b(\ell-k)\frac{u}{u-1}\qquad(u\in\R,\ u>1).
\]
Then $|G|=mq\le f(q)$. The second derivative satisfies
\[
 f''(u)=\frac{4b(\ell-k)}{(u-1)^3}>0.
\]
By \eqref{eq:index-range}, $q\in[2,2b+1]$.
The function $f$ is convex on this interval and attains its maximum at an endpoint.
The endpoint values are
\[
 f(2)=4k+4b(\ell-k),\qquad
 f(2b+1)=(2b+1)(k+\ell).
\]
Their difference is
\[
 f(2b+1)-f(2)=(2b-1)(3k-\ell)>0,
\]
since $b\ge1$ and $\ell\le2k$.
Combining $|G|\le f(q)$ with the endpoint comparison, we obtain
\[
 |G|\le(2b+1)(k+\ell).\qedhere
\]
\end{proof}

\section{The case \texorpdfstring{$\ell=2k$}{ell=2k}}\label{sec:critical}

Let $\mathcal{P}\subset\R^d$ be a $d$-dimensional lattice polytope with
\[
 h^*_{\mathcal{P}}(t)=1+at^k+bt^{2k},
 \qquad a,b\ge1,\quad k\ge1.
\]
We prove Theorem~\ref{thm:critical} in this section.
The preparatory arguments below concern the case $k\ge2$ and $m>2k$.
Here $\mathcal{P}$ is an empty lattice simplex. Write $\Delta=\mathcal{P}$ and retain the notation of Sections~\ref{sec:group} and~\ref{sec:general-bound}, with $\ell=2k$.
The index satisfies \eqref{eq:index-range}, and Lemma~\ref{lem:uniform-coset} gives $L(C)>0$ for every nonzero coset.

Substituting $\ell=2k$ in \eqref{eq:cyclic-L-general} and \eqref{eq:code-L-general}, we obtain, for every nonzero coset $C\in Q$,
\begin{equation}\label{eq:critical-cyclic-L}
 L(C)=\frac{m\nu(C)+\tau(C)}k-2
\end{equation}
in the cyclic case, and
\begin{equation}\label{eq:critical-code-L}
 L(C)=2\left(\frac p\lambda\nu(C)+\frac{\tau(C)}\lambda-1\right)
\end{equation}
in the simplex-code case.

In the simplex-code case with $p$ odd, \eqref{eq:code-relations} gives
\[
 \frac{|S|}{k}=\frac{2(p^r-1)}{p^{r-1}(p-1)}
 <\frac{2p}{p-1}\le3.
\]
Thus $|S|<3k=k+\ell$, so Lemma~\ref{lem:external-injective} applies.
In particular, $\nu(C)\ge1$ for every nonzero coset $C\in Q$.

We next bound the number of elements with only one nonzero coordinate.
For a subgroup $A\le\T^N$, we say that $A$ is \emph{supported on a single coordinate} if there is an index $i_0\in\{1,\ldots,N\}$ such that $z_j=0$ for every $z\in A$ and every $j\ne i_0$.
The Hamming weight of $z\in\T^N$ is the number of its nonzero coordinates.

\begin{lemma}\label{lem:weight-one}
Let $A$ be a finite subgroup of $\T^N$, with $|A|=q$.
If $A$ is not supported on a single coordinate, then at most $q/2$ of its nonzero elements have Hamming weight one.
\end{lemma}

\begin{proof}
For each $i\in\{1,\ldots,N\}$, let
\[
 A_i=\{z\in A:z_j=0\text{ for every }j\ne i\}.
\]
The weight-one elements form the disjoint union of the sets $A_i\setminus\{0\}$.
If all $A_i$ are trivial, there is nothing to prove.
If precisely one is nontrivial, it is proper in $A$ by the hypothesis, and its order is at most $q/2$.

Suppose there are $s\ge2$ nontrivial subgroups among the $A_i$.
After relabelling the coordinates, denote them by $A_1,\ldots,A_s$ and put $q_i=|A_i|\ge2$ for $1\le i\le s$.
These subgroups are supported on distinct coordinates, so their sum is a direct sum contained in $A$, and $\prod_{i=1}^s q_i\le q$.
We claim that
\[
 \sum_{i=1}^s(q_i-1)\le\frac12\prod_{i=1}^s q_i.
\]
This inequality follows by induction on $s$.
Indeed, for $s=2$, it is equivalent to $(q_1-2)(q_2-2)\ge0$.
Suppose $s\ge3$ and set $Q_0=\prod_{i=1}^{s-1}q_i$.
By the induction hypothesis and $Q_0\ge4$, we obtain
\[
 \begin{aligned}
 \sum_{i=1}^s(q_i-1)
 &\le\frac{Q_0}{2}+(q_s-1)\\
 &\le\frac{Q_0}{2}+\frac{Q_0}{2}(q_s-1)
 =\frac12\prod_{i=1}^s q_i.
 \end{aligned}
\]
It follows that the number of weight-one elements is at most $\frac12\prod_{i=1}^s q_i\le q/2$.
\end{proof}

The size of the external support gives stronger bounds in two of the cases of Proposition~\ref{prop:BH}.

\begin{proposition}\label{prop:structural-linear}
Assume $\ell=2k$, $k\ge2$, and $m>2k$.
\begin{enumerate}
\item\label{item:cyclic-linear}
In the cyclic case, one has
\[
 \vol(\mathcal{P})\le4(b+1).
\]
\item\label{item:large-prime-linear}
In the simplex-code case with $p\ge5$, one has
\[
 \vol(\mathcal{P})\le8b-2.
\]
\end{enumerate}
\end{proposition}

\begin{proof}
Fix a nonzero coset $C\in Q$.
Lemma~\ref{lem:uniform-coset} gives $L(C)>0$, so there is an element
$y\in B\cap(C\cup(-C))$.
Since $\htop(y)=2k$ and $\htop(-y)\ge k$, equation \eqref{eq:support-height} gives $|\supp(y)|\ge3k$.
All elements of a coset have the same external coordinates, and $C$ and $-C$ have the same external support.
Consequently,
\[
 \nu(C)=|\supp(y)\setminus S|\ge3k-|S|.
\]
This holds for every nonzero coset.

In the cyclic case, $|S|=2k$ by Proposition~\ref{prop:BH}(i), so $\nu(C)\ge k$.
Using $\tau(C)\ge0$ in \eqref{eq:critical-cyclic-L}, we obtain
\[
 L(C)\ge m-2\qquad(C\ne H).
\]
Summing over the $q-1$ nonzero cosets and using \eqref{eq:L-sum} gives
\[
 2b\ge(q-1)(m-2).
\]
Since $q\ge2$ by \eqref{eq:index-range} and $m>2k\ge4$, it follows that
\[
 \begin{aligned}
 4(b+1)-mq
 &\ge2(q-1)(m-2)+4-mq\\
 &=(q-2)(m-4)\ge0.
 \end{aligned}
\]
Hence $\vol(\mathcal{P})=mq\le4(b+1)$.

Now suppose that $H$ is in the simplex-code case with $p\ge5$.
By \eqref{eq:code-relations},
\[
 \frac{|S|}{k}=\frac{2(p^r-1)}{p^{r-1}(p-1)},
 \qquad m=p^r.
\]
The external-support estimate and \eqref{eq:critical-code-L} yield
\[
 \begin{aligned}
 L(C)
 &\ge\frac{m}{k}(3k-|S|)-2\\
 &=\frac{(p-3)m+2}{p-1}
 \ge\frac{m+1}{2}.
 \end{aligned}
\]
The last inequality is equivalent to $(p-5)(m-1)\ge0$.
Summing and applying \eqref{eq:L-sum}, we obtain
\[
 4b\ge(q-1)(m+1).
\]
Finally, $q\ge2$ gives
\begin{align*}
 \vol(\mathcal{P})=mq
 &\le2m(q-1)\\
 &=2(q-1)(m+1)-2(q-1)\\
 &\le8b-2(q-1)\le8b-2.\qedhere
\end{align*}
\end{proof}

\begin{proof}[Proof of Theorem~\ref{thm:critical}]
Assume $b\ge2$.
If $k=1$, then Treutlein's theorem gives $a\le3b+3$.
Therefore
\[
 \vol(\mathcal{P})=a+b+1\le4(b+1),
\]
with equality if and only if $a=3b+3$.
The rectangles given after Theorem~\ref{thm:critical} attain equality for every $b\ge2$.

Suppose $k\ge2$. As observed in Section~\ref{sec:general-bound}, $\mathcal{P}$ is an empty lattice simplex.
Use the notation introduced above, with $\ell=2k$.
If $m\le2k$, then \eqref{eq:index-range} gives
\[
 \vol(\mathcal{P})=mq\le2k(2b+1)<4k(b+1).
\]
We may therefore assume $m>2k$.
We show that $|G|<4k(b+1)$ in each case of Proposition~\ref{prop:BH}.
We use the following consequence of \eqref{eq:L-sum}. If $L(C)\ge2$ for all $C\in Q\setminus\{H\}$, then
\[
 2b=\sum_{C\in Q\setminus\{H\}}L(C)\ge2(q-1),
 \qquad\text{hence}\qquad q\le b+1.
\]

\smallskip
\noindent\emph{The cyclic case.}
Proposition~\ref{prop:structural-linear}(i) gives
\[
 |G|\le4(b+1)<4k(b+1),
\]
since $k\ge2$.

\smallskip
\noindent\emph{The simplex-code case with $p=2$.}
By \eqref{eq:lambda-bound}, the positive integer $\lambda$ satisfies $\lambda<2$, so $\lambda=1$ and $m=4k$.
For every nonzero coset $C$, equation \eqref{eq:critical-code-L} gives
\[
 L(C)=2\bigl(2\nu(C)+\tau(C)-1\bigr).
\]
Since $L(C)>0$ by Lemma~\ref{lem:uniform-coset}, it is a positive even integer and thus at least two.
It follows that $q\le b+1$ and
\[
 |G|=4kq\le4k(b+1).
\]
Suppose that equality holds.
Then $q=b+1$, and \eqref{eq:L-sum} forces $L(C)=2$ for every nonzero coset.
Consequently,
\[
 2\nu(C)+\tau(C)=2\qquad(C\ne H).
\]
Suppose a nonzero coset $C$ satisfied $(\nu(C),\tau(C))=(0,2)$. Choose $x\in G$ with $C=x+H$. Then $2x\in G$ would be nonzero with support of size two, contradicting \eqref{eq:min-support} since $2k\ge4$.
Thus $(\nu(C),\tau(C))=(1,0)$ for every nonzero coset.
The external projection is injective on $Q$, because no nonzero coset lies in its kernel.
Every nonzero element of its image has Hamming weight one, so the image is supported on a single coordinate.
Projection onto that coordinate embeds $Q$ in $\T$, so $Q$ is cyclic.
Choose a generator $C_0\in Q$ and a representative $x_0\in G$ with $C_0=x_0+H$.
Since $\tau(C_0)=0$, the element $2x_0\in G$ vanishes on $S$ and has at most one nonzero external coordinate.
Equation~\eqref{eq:min-support} gives $2x_0=0$.
Hence $2C_0=H$ and $q=2$, so $b=q-1=1$, contradicting $b\ge2$.

\smallskip
\noindent\emph{The simplex-code case with $p\ge5$.}
Proposition~\ref{prop:structural-linear}(ii) gives
\[
 |G|\le8b-2<4k(b+1),
\]
again because $k\ge2$.

\smallskip
\noindent\emph{The simplex-code case with $p=3$.}
By Proposition~\ref{prop:BH}(ii), $\lambda$ is a positive even integer, and \eqref{eq:lambda-bound} gives $\lambda<3$.
Thus $\lambda=2$ and $m=3k$.
As observed at the beginning of this section, $|S|<3k=k+\ell$, so Lemma~\ref{lem:external-injective} gives $\nu(C)\ge1$ for every nonzero coset.
Substituting $p=3$ and $\lambda=2$ into \eqref{eq:critical-code-L}, we obtain
\begin{equation}\label{eq:p3-L}
 L(C)=3\nu(C)+\tau(C)-2
\end{equation}
for every nonzero coset $C\in Q$.
Set $E=\{C\in Q\setminus\{H\}:L(C)=1\}$.
Since $\nu(C)$ and $\tau(C)$ are integers with $\nu(C)\ge1$ and $\tau(C)\ge0$, \eqref{eq:p3-L} gives
\[
 C\in E\quad\Longleftrightarrow\quad \nu(C)=1\text{ and }\tau(C)=0
 \qquad(C\ne H).
\]
We claim that
\begin{equation}\label{eq:E-bound}
 |E|\le q/2.
\end{equation}
The external projection $\overline\rho$ is injective by Lemma~\ref{lem:external-injective}, so $|\overline\rho(Q)|=q$.
Every $C\in E$ has $\nu(C)=1$, and thus $\overline\rho(C)$ has Hamming weight one.
If $\overline\rho(Q)$ is not supported on a single coordinate, Lemma~\ref{lem:weight-one} applied to this subgroup gives $|E|\le q/2$.

Suppose instead that $\overline\rho(Q)$ is supported on a single coordinate.
Let
\[
 T=\{x+H:x\in G,\ 3x_i=0\text{ for every }i\in S\}.
\]
This is the kernel of the homomorphism $Q\to\T^S$ given by $x+H\mapsto(3x_i)_{i\in S}$. The map is well-defined because $3H=0$.
Since $\overline\rho$ is injective, we have
$\overline\rho(C)\ne0$ for every nonzero coset $C\in Q$.
By our assumption that $\overline\rho(Q)$ is supported on a single coordinate,
it follows that $\nu(C)=1$ for every $C\ne H$.

On the other hand, the definition of $T$ gives
$C\in T$ if and only if $\tau(C)=0$.
Since $C\in E$ is equivalent to $\nu(C)=1$ and $\tau(C)=0$
for $C\ne H$, we conclude that $E=T\setminus\{H\}$.
If $T\ne Q$, then $T$ is a proper subgroup of $Q$,
so $|T|\le q/2$. Consequently,
\[
 |E|=|T|-1<q/2.
\]
If $T=Q$, projection onto the unique possible nonzero external coordinate embeds $Q$ in $\T$, so $Q$ is cyclic.
Choose a generator $C_0=x_0+H$ with $x_0\in G$.
As $C_0\in T$, the element $3x_0\in G$ vanishes on $S$ and has at most one nonzero external coordinate.
Equation~\eqref{eq:min-support} forces $3x_0=0$, so $3C_0=H$ and $q=3$.
Both nonzero cosets belong to $E$, and \eqref{eq:L-sum} gives $2b=2$, contrary to $b\ge2$.
This establishes \eqref{eq:E-bound}.

Every nonzero coset outside $E$ has $L(C)\ge2$.
Therefore
\[
 \begin{aligned}
 2b&\ge|E|+2(q-1-|E|)\\
 &=2q-2-|E|\ge\frac{3q}{2}-2.
 \end{aligned}
\]
Hence $3q\le4(b+1)$, and
\[
 |G|=3kq\le4k(b+1).
\]
Suppose that equality holds.
Then $3q=4(b+1)$, $|E|=q/2$, and $L(C)=2$ for every nonzero coset outside $E$.
Equation \eqref{eq:p3-L} and $\nu(C)\ge1$ show that these cosets satisfy $(\nu(C),\tau(C))=(1,1)$, whereas those in $E$ satisfy $(\nu(C),\tau(C))=(1,0)$.
Let $T_0$ be the kernel of the homomorphism $Q\to\T^S$ given by $x+H\mapsto(3x_i)_{i\in S}$.
Then $E=T_0\setminus\{H\}$, so $|T_0|=q/2+1$.
Since $3q=4(b+1)$ and $b\ge2$, we have $q\ge4$, giving
\[
 \frac q2<|T_0|=\frac q2+1<q.
\]
This is impossible because a proper subgroup of $Q$ has order at most $q/2$.

Thus $|G|<4k(b+1)$ whenever $k\ge2$.
Since the normalized volume is an integer, $\vol(\mathcal{P})\le4k(b+1)-1$.
Together with the case $k=1$, this proves the theorem.
\end{proof}

The next proposition gives a complementary estimate depending only on $b$.
Its proof uses the height distribution in $G$, without any restriction on the order of the subgroup $H$.

\begin{proposition}\label{prop:b-only}
Let $\mathcal{P}\subset\R^d$ be a $d$-dimensional lattice polytope with
\[
 h^*_{\mathcal{P}}(t)=1+at^k+bt^{2k},
 \qquad a\ge1,\quad b\ge2,\quad k\ge2.
\]
Then $\mathcal{P}$ is an empty lattice simplex $\Delta$.
The group $G=\Lambda_\Delta$ is generated by
$B=\{x\in G:\htop(x)=2k\}$. Moreover, the normalized volume of $\mathcal{P}$ satisfies
\[
 \vol(\mathcal{P})\le(8b^2)^b.
\]
\end{proposition}

\begin{proof}
Since $h^*_1=0$, the polytope is an empty simplex.
Put $G=\Lambda_\Delta$ and $J=\langle B\rangle$, where $|B|=b$.
We first prove that $J=G$.
For each subgroup $A\le G$ and each $i\in\{0,\ldots,d\}$, let
\[
 d_i(A)=|\{x_i:x\in A\}|.
\]
The image of the $i$th coordinate projection on $A$ is $\T[d_i(A)]$.
Each element of this image is taken exactly $|A|/d_i(A)$ times, since its inverse image is a coset of the kernel.
Using representatives in $[0,1)$, we obtain
\[
 \begin{aligned}
 \frac1{|A|}\sum_{x\in A}\htop(x)
 &=\sum_{i=0}^d\frac1{d_i(A)}
      \sum_{j=0}^{d_i(A)-1}\frac{j}{d_i(A)}\\
 &=\frac12\sum_{i=0}^d\left(1-\frac1{d_i(A)}\right).
 \end{aligned}
\]
Because $J\le G$, the coordinate images satisfy
$\{x_i:x\in J\}\subseteq\{x_i:x\in G\}$, so $d_i(J)\le d_i(G)$.
Consequently,
\[
 \frac1{|J|}\sum_{x\in J}\htop(x)
 \le \frac1{|G|}\sum_{x\in G}\htop(x).
\]
All $b$ elements of height $2k$ belong to $J$, and every other nonzero element has height $k$.
Thus the last inequality is
\[
 k\left(1+\frac{b-1}{|J|}\right)
 \le k\left(1+\frac{b-1}{|G|}\right).
\]
Since $b\ge2$, this implies $|J|\ge|G|$.
Together with $J\le G$, we conclude that $J=G$.

Next fix $x\in B$ and let $N$ be its order in $G$.
By \eqref{eq:support-height}, we have $|\supp(x)|\ge3k$.
If $1\le j<N$ and $\gcd(j,N)=1$, then
\[
 \supp(jx)=\supp(x).
\]
Indeed, an integer multiple cannot create a nonzero coordinate from a zero one, so $\supp(jx)\subseteq\supp(x)$.
Since $\gcd(j,N)=1$, choose an integer $v$ such that $vj\equiv1\pmod N$.
Then $v(jx)=x$, because $Nx=0$.
Thus $x$ is an integer multiple of $jx$, giving $\supp(x)\subseteq\supp(jx)$.
If neither $jx$ nor $-jx$ belonged to $B$, both would have height $k$, contradicting $|\supp(jx)|\ge3k$.
Hence every generator of $\langle x\rangle$ belongs to $B\cup\{-y:y\in B\}$.
Let $\phi$ denote Euler's totient function.
The preceding observation gives
\[
  \phi(N)\le2b.
\]
Writing $N=\prod_{j=1}^u p_j^{e_j}$ for the prime factorization of $N$, we obtain
\[
 \frac{\phi(N)^2}{N}
 =\prod_{j=1}^u p_j^{e_j-2}(p_j-1)^2\ge\frac12.
\]
Every factor is at least $1$, except possibly the factor with $p_j=2$ and $e_j=1$, which equals $1/2$.
It follows that $N\le2\phi(N)^2\le8b^2$.

Finally, enumerate $B=\{x_1,\ldots,x_b\}$ and let $N_j$ be the order of $x_j$.
Since $G$ is abelian and $G=\langle B\rangle$, the homomorphism
\[
 \prod_{j=1}^b\Z/N_j\Z\longrightarrow G,
 \qquad(c_1,\ldots,c_b)\longmapsto\sum_{j=1}^b c_jx_j
\]
is surjective.
Therefore
\[
 \vol(\mathcal{P})=|G|
 \le\prod_{j=1}^b N_j\le(8b^2)^b.\qedhere
\]
\end{proof}

The condition $b\ge2$ is used when comparing the average heights on $J$ and $G$.
For $b=1$, that comparison gives no restriction on $[G:J]$, and Proposition~\ref{prop:b-one} shows that the volume is not bounded independently of $k$.

\begin{remark}\label{rem:general-b-only}
The proof of Proposition~\ref{prop:b-only} also applies to a lattice polytope $\mathcal{P}$ with
\[
 h^*_{\mathcal{P}}(t)=1+at^k+bt^\ell,
 \qquad a,b\ge1,\quad 2\le k<\ell\le2k,
\]
provided that $(\ell-k)b>k$.
Indeed, $\mathcal{P}$ is an empty simplex $\Delta$.
Put $G=\Lambda_\Delta$, $B=\{x\in G:\htop(x)=\ell\}$, and $J=\langle B\rangle$.
The same comparison of average heights gives
\[
 k+\frac{(\ell-k)b-k}{|J|}
 \le k+\frac{(\ell-k)b-k}{|G|}.
\]
The assumed positivity of $(\ell-k)b-k$ implies $J=G$.
For $x\in B$, the bound $|\supp(x)|\ge k+\ell>2k$ gives the same estimate $\phi(\operatorname{ord}(x))\le2b$.
The rest of the proof is unchanged, and hence $\vol(\mathcal{P})\le(8b^2)^b$.
\end{remark}

For $k,b\ge2$, Theorem~\ref{thm:critical} and Proposition~\ref{prop:b-only} together give
\[
 \vol(\mathcal{P})\le
 \min\{4k(b+1)-1,\,(8b^2)^b\}.
\]
We conclude with an example that gives a lower bound for any upper bound depending only on $b$.

\begin{example}\label{ex:replicated-triangles}
Let $\Delta_b=\operatorname{conv}((0,0),(2b+2,0),(0,2))$.
The triangle $\Delta_b$ has $b$ interior lattice points and $2b+6$ boundary lattice points.
Thus the standard formulas for lattice polygons give
$h^*_{\Delta_b}(t)=1+(3b+3)t+bt^2$,
and hence $|\Lambda_{\Delta_b}|=4(b+1)$.
For $k\ge1$, consider the subgroup
\[
 G_b^{(k)}=\{(x,\ldots,x):x\in\Lambda_{\Delta_b}\}\subseteq\T^{3k},
\]
where each $x$ is repeated $k$ times.
This operation multiplies every height by $k$ while preserving the order of the group.
By the correspondence in \cite[Theorem~2.3 and Proposition~2.6]{BHCode},
there is a $(3k-1)$-dimensional lattice simplex $\Delta_{k,b}$ with
\[
 h^*_{\Delta_{k,b}}(t)=1+(3b+3)t^k+bt^{2k},
 \qquad \vol(\Delta_{k,b})=4(b+1).
\]
\end{example}

In particular, the value $a=3b+3$ from Scott's inequality is realized for every $k\ge1$.
For $b\ge2$, Example~\ref{ex:replicated-triangles} shows that an upper bound depending only on $b$ must be at least $4(b+1)$.
Its volume is a factor of $k$ below the upper bound $4k(b+1)$ in Theorem~\ref{thm:critical}.
Neither this construction nor Proposition~\ref{prop:b-only} determines the optimal bound for fixed $b$.
The preceding estimates leave the following question.

\begin{question}\label{que:uniform-linear}
Does there exist an absolute constant $c>0$ such that
\[
 \vol(\mathcal{P})\le c(b+1)
\]
for every lattice polytope $\mathcal{P}$ with
\[
 h^*_{\mathcal{P}}(t)=1+at^k+bt^{2k},
 \qquad a\ge1,\quad b\ge2,\quad k\ge2?
\]
\end{question}

Example~\ref{ex:replicated-triangles} forces $c\ge4$.
Unlike Proposition~\ref{prop:b-only}, the question asks for a single linear bound as both $k$ and $b$ vary.
With $H$ chosen as in Section~\ref{sec:group}, Proposition~\ref{prop:structural-linear} settles the cyclic case with $m>2k$ and the simplex-code case with $p\ge5$.
It remains to consider $m\le2k$ and the simplex-code cases with $p=2$ or $p=3$ when $m>2k$.
In these remaining simplex-code cases, Proposition~\ref{prop:BH}(ii) and \eqref{eq:lambda-bound} give $\lambda=1$ and $k=2^{r-2}$ for $p=2$, and $\lambda=2$ and $k=3^{r-1}$ for $p=3$.
Thus, if $k$ is neither a power of $2$ nor a power of $3$, only the case $m\le2k$ remains.

\subsection*{Declaration on the use of generative AI}
OpenAI's ChatGPT was used for exploratory discussions of proof
strategies, English-language editing, and \LaTeX{} assistance.
All mathematical arguments and references were verified by the author.

\end{document}